\documentclass[a4paper,12pt,leqno]{amsart}
\usepackage{amsmath,amssymb,amsxtra,comment,graphicx,psfrag}
\usepackage{bm,mathrsfs}
\usepackage{mathtools}
\usepackage{xspace}
\usepackage{color}
\usepackage{enumerate}
\usepackage{hyperref}
\usepackage{pdfsync}
\usepackage{hhline}
\usepackage[margin=1.3in]{geometry}

\allowdisplaybreaks
\definecolor{blue}{rgb}{0,0.0,0.9}

\def\eps{\varepsilon}

\def\C{{\mathbb C}}
\def\R{{\mathbb R}}

\def\d{{\mathrm d}}

\def\tr|{|\!|\!|}

\theoremstyle{plain}
 \newtheorem{theorem}{Theorem}[section]
 \newtheorem{lemma}{Lemma}[section]

 \theoremstyle{definition}

 \newtheorem{remark}{Remark}[section]

\numberwithin{equation}{section}
\numberwithin{table}{section}
\numberwithin{figure}{section}

\newcommand{\QED}{\hfill $\square$}

\begin{document}

\title[]
{Runge--Kutta time discretization of\\ nonlinear parabolic equations studied via discrete maximal parabolic regularity} 

%
%
\author[]{Peer C. Kunstmann}
\address{Institut f\"ur Analysis, KIT, 76131 Karlsruhe, Germany.} 
\email {\sf{peer.kunstmann@kit.edu} }

\author[]{Buyang Li}
\address{Department of Applied Mathematics,  
           The Hong Kong Polytechnic University,  
           Hung Hom, Kowloon, Hong Kong.} 
\email {\sf{buyang.li@polyu.edu.hk} }

\author[]{Christian Lubich}
\address{Mathematisches Institut, Universit\"at T\"ubingen, Auf der Morgenstelle, 
D-72076 T\"ubingen, Germany} 
\email {\sf{lubich{\it @}na.uni-tuebingen.de} }



\maketitle

\begin{abstract}
For a large class of fully nonlinear parabolic equations, which include gradient flows for energy functionals that depend on the solution gradient, the semidiscretization in time by implicit Runge--Kutta methods such as the Radau IIA methods of arbitrary order is studied. Error bounds are obtained in the $W^{1,\infty}$ norm uniformly on bounded time intervals and, with an improved approximation order, in the parabolic energy norm. The proofs rely on discrete maximal parabolic regularity. This is used to obtain $W^{1,\infty}$  estimates, which are the key to the numerical analysis of these problems.\\

\noindent Keywords. Runge--Kutta method, maximal parabolic regularity, 
nonlinear parabolic equation, gradient flow, stability, error bounds. \\

\noindent AMS subject classification. 65M12, 65M60. 
\end{abstract}




\maketitle

\section{Introduction}\label{Se:intr}
This paper is concerned with the stability and error analysis of implicit Runge--Kutta time discretizations of nonlinear parabolic initial-boundary value problems for $u=u(x,t)$,
\begin{equation}
\label{ivp}
\frac{\partial u}{\partial t}  
=
\nabla\cdot f(\nabla u,u) , 
\qquad  x\in\varOmega,  \ 0<t\le T,
\end{equation}
on a given bounded smooth domain $\varOmega\subset \R^d$ of arbitrary dimension $d\ge 1$
and for a given final time $T>0$, taken with homogeneous Dirichlet boundary conditions $u=0$ on $\partial\varOmega\times[0,T]$
and with given initial data $u(\cdot,0)=u_0$ on $\varOmega$.

The flux function
$f:\R^d\times \R\rightarrow \R^d$ is assumed to be a smooth function  
satisfying a local ellipticity condition: for every $(p,u)\in\R^d\times \R$, 
the matrix 
\begin{equation}\label{locally-elliptic}
\text{$\partial_p f(p,u)\in\R^{d\times d}$ has a positive definite symmetric part.}
\end{equation}
We do not require uniform ellipticity: some eigenvalues of the symmetric part  $\frac12 (\partial_p f(p,u)+\partial_p f(p,u)^T)$ may tend to $0$ or $+\infty$ as $|(p,u)|\to\infty$.

We will, however, assume that the initial-boundary value problem admits a sufficiently regular solution, and we ask for stability and rates of convergence of time discretizations in this case.
\medskip

The problem \eqref{ivp} occurs in many applications, such as  the following where actually $f(p,u)=f(p)$ does not depend on $u$:
\begin{itemize}


%

\item minimal surface flow \cite{LT78,SS11} 
and the regularized models of total variation flow \cite{FOP05,FP03,LS14}, 
where 
$$
f(p)=\frac{p}{\sqrt{\lambda^2+|p|^2}}. 
$$

\item 
More generally, with $f(p)=\nabla_p F(p)$ for a smooth convex function $F:\R^d\rightarrow\R$,
\eqref{ivp} appears
as the $L^2(\varOmega)$ gradient flow, 
$$
(\partial_t u,v)_{L^2(\varOmega)}=-E'(u)v  \qquad\text{for all $v$ in a dense subspace of $H^1_0(\varOmega)$},
$$
for the energy functional
$E(u) = \int_\varOmega F(\nabla u) \d x$; see, e.g., \cite[Section 9.6.3]{Evans}.
\end{itemize}
The problem \eqref{ivp} also includes quasilinear equations, where $f(p,u)=A(u)p$ with a positive definite matrix $A(u)$, which may degenerate as $|u|\to\infty$.

Due to the strong nonlinearity of the equation, existing works on 
error estimates of the time discretization of \eqref{ivp} are very limited. 
Feng and Prohl \cite{FP03} have 
proved optimal-order convergence rate of the 
finite element solution of the regularized total variation flow  
with an implicit backward Euler scheme, 
under the time stepsize restriction $\tau=o(h^2)$, 
which was used to control the numerical solution in the $W^{1,\infty}$ norm 
via the inverse inequality. 
Convergence of the numerical solution 
was proved in \cite{FOP05} without time-step size restriction, 
without explicit convergence rate. 
By using the methodology of \cite{LS13}, 
Li and Sun presented optimal-order $L^2$-norm error estimates for 
the finite element solution of the minimal surface flow 
with  a linearized semi-implicit backward Euler scheme, without 
restriction on the time stepsize \cite{LS14}. Since their proof is based on the 
$L^2$-norm error estimate, they have assumed that the order of finite 
elements are greater than one in order to control the 
$W^{1,\infty}$ norm of the numerical solution via the inverse inequality. 
Overall, due to the strong nonlinearity of the equation, 
error estimates for the numerical solution of \eqref{ivp} need  
uniform boundedness of the numerical solution in $W^{1,\infty}$ 
(the bound should be independent of the mesh size), 
which is the main difficulty of this problem. 
Existing works on the problem 
are all restricted to backward Euler time discretization. 

In this paper we study semidiscretization in time by implicit Runge--Kutta methods such as the collocation methods based on the Radau nodes, which have excellent stability properties, allow for arbitrarily high order and can be implemented efficiently \cite[Chapter IV]{HW}. To emphasize the basic techniques and to keep the paper at a reasonable length, we do not include the effect of space discretization by finite elements in our stability and error analysis.
We note, however, that in considering only time discretization we cannot use inverse estimates, which are often convenient, but are restricted to quasi-uniform meshes and moreover lead to restrictions as indicated in the previous paragraph. It is thus of interest to develop techniques that do not rely on inverse estimates. Our results are new even for the case of the backward Euler time discretization.
This paper may provide a foundation for further analysis of fully discrete approximations of the problem.

In Section 2 we describe the temporal semidiscretization by implicit Runge--Kutta methods and present our main results, which are error bounds in the $W^{1,\infty}$ norm and, with a higher approximation order, in the energy norm. The proof of these results forms the remainder of the paper.

Section 3 presents a sequence of auxiliary results related to maximal $L^p$ regularity, which is the basic technique for obtaining our stability and error bounds. Discrete maximal $L^p$ regularity was first shown for the backward Euler method by Ashyralyev, Piskarev \& Weis \cite{APW} and for higher-order A-stable (and $A(\alpha)$-stable) multistep and Runge--Kutta time discretizations by Kov\'acs, Li \& Lubich \cite{KLL}. Discrete maximal $L^p$ regularity up to a factor logarithmic in the stepsize was given by Leykekhman \& Vexler \cite{LV} for discontinuous Galerkin time discretizations. The above-mentioned results relate to linear problems.
Discrete maximal $L^p$ regularity  was applied to the error analysis of time discretizations of reaction-diffusion equations in \cite{KLL}, of Ginzburg-Landau equations in \cite{Li-GL}, and of quasilinear parabolic equations in \cite{ALL}.

The proof of the error bound in the $W^{1,\infty}$ norm is given in Section 4, that of the improved error bound in the energy norm in Section 5.

\section{Runge--Kutta time discretization and statement of the main results}\label{Se:statement}

We consider the time discretization of \eqref{ivp} with constant stepsize $\tau>0$ (this could be relaxed to a fixed number of changes of the stepsize) by an implicit Runge-Kutta method with properties that are, in particular, satisfied by the $s$-stage Radau IIA method \cite[Section IV.5]{HW}, which is the collocation method at the Radau nodes (with right-most node $c_s=1$) and can also be viewed as a fully discretized discontinuous Galerkin method in time \cite{AMN}. We require the following properties (cf.~\cite[Section IV.3]{HW} for these notions):
\begin{align}
&\text{The Runge--Kutta method is A-stable, } \nonumber \\[-1mm]
&\text{it has an invertible coefficient matrix  $(a_{ij})_{i,j=1}^s$} 
\label{rk-cond}\\[-1mm]
&\text{and its weights satisfy $b_j=a_{s,j}$ ($j=1,\dots,s$). }
\nonumber
\end{align}
We let $t_n=n\tau$ for $n\ge 0$ (as long as $t_n$ does not exceed the final time $T$) and set $t_{n,i}=t_n +c_i\tau$, where $c_i=\sum_{j=1}^s a_{ij}$ are the nodes of the Runge--Kutta method, with $c_s=1$ so that $t_{n+1}=t_{n,s}$.

We denote by $u_{n,i}$ ($i=1,\dots,s$) the internal stages and by $u_{n}$ the solution approximation at the grid point $t_{n}$. The last condition in \eqref{rk-cond} ensures that
\begin{equation}\label{last-stage}
u_{n+1}=u_{n,s}.
\end{equation}
The time discretization of \eqref{ivp} is then determined by the equations 
\begin{equation}\label{RK}
u_{n,i}= u_{n}+\tau\sum_{j=1}^s
a_{ij} \nabla \cdot f(\nabla u_{n,j},u_{n,j}) \qquad (i=1,\dots,s)
\end{equation}
together with the  Dirichlet boundary conditions $u_{n,i}=0$ on $\partial\varOmega$. These equations are to be solved subsequently for $n=0,1,2,\dots$

\begin{remark}
Further finite element discretization of \eqref{RK} can be done 
in the following way: find $u_{n,i}^{h}$ in the finite element space $S_h$ such that 
$$
(u_{n,i}^{h},v_h)= (u_{n}^{h},v_h)-\tau\sum_{l=1}^s
a_{ij} \Bigl( f(\nabla u_{n,j}^{h},u_{n,j}^{h}) ,\nabla v_h\Bigr)
\qquad \forall\, v_h\in S_h ,
$$
and $u_{n+1}^h=u_{n,s}^h$. For the efficient implementation of the fully discrete Runge--Kutta equations, using systems of linear equations of just the dimension of $S_h$, we refer to \cite[Section IV.8]{HW}.
In this paper, we focus on the time discretization \eqref{RK}. 
\end{remark}

We recall the notion of {\it stage order}, cf.~\cite[p.\,226]{HW}: The Runge--Kutta method has stage order $k$ if for each $i=1,\dots,s$,
\begin{equation}\label{stage-order}
\sum_{j=1}^s a_{ij} c_j^{l-1} = \frac{c_i^l}l, \qquad\  l=1,\dots, k.
\end{equation}
In particular, the stage order of the $s$-stage Radau IIA method (as of any collocation method with polynomials of degree $s$) is $k=s$.

The stage order determines to what order the internal stages $u_{n,i}$ approximate the exact solution values $u(t_{n,i})$, and to what order the derivative approximations
\begin{equation}\label{RKDerivative}
\dot u_{n,j} :=  \nabla \cdot f(\nabla u_{n,j},u_{n,j}), \qquad j=1,\dots, s,
\end{equation}
approximate the exact solution derivatives $\partial_t u(t_{n,j})$, provided the solution is sufficiently regular in time.

To simplify the notation, we define the following 
vectors:
\begin{align} \label{vec-not}
&\vec u_n:=(u_{n,i})_{i=1}^s, 
&\vec{\dot u}_n:=(\dot u_{n,i})_{i=1}^s, 
 \\
&u(\vec t_n):=(u(t_{n,i}))_{i=1}^s,
&\vec t_n:=(t_{n,i})_{i=1}^s  .
\end{align}  

We can now state our first main result, which in particular controls the 
$W^{1,\infty}(\varOmega)$-norm of the internal stages uniformly over the bounded time interval. 

\begin{theorem}\label{thm:main} Consider a Runge--Kutta method of stage order $k$ that satisfies \eqref{rk-cond}, such as the Radau IIA method with $s=k$ stages.
Assuming that the solution $u$ of \eqref{ivp} is sufficiently regular,  i.e., 
\begin{align}\label{Regularity}
u\in C^{k+1}\big ([0,T];L^q(\varOmega)\big ) 
\cap C\big ([0,T];W^{2,q}(\varOmega)\big )  , \quad\text{for some $q>d$,}
\end{align}
there exists a positive constant $\tau_1$ such that for
$\tau<\tau_1$ the discrete problem \eqref{RK} admits a unique solution
that satisfies 
\begin{subequations}\label{err-est}
\begin{align}
\label{est1} 
\max_{0\leq n\leq N}
\left(
\|\vec u_{n}-u(\vec t_{n})\|_{L^\infty(\varOmega)^s}
+\|\nabla \vec u_{n}-\nabla u(\vec t_{n})\|_{L^\infty(\varOmega)^{ds}}
\right)
\le C\tau^k ,\\
\label{est2}
\bigg(\sum_{n=0}^N 
\tau \|\vec{\dot u}_{n}-\partial_t u(\vec t_{n})\|^p_{L^q(\varOmega)^s} 
+ 
\sum_{n=0}^N \tau \|\vec u_{n}- u(\vec t_{n})\|^p_{W^{2,q}(\varOmega)^s}\bigg)^{\frac1p}
\le C_{p,q}\tau^k ,\\
\quad\hbox{for all }\ 1<p<\infty . \nonumber 
\end{align}
\end{subequations}
The constants $C$ and $C_{p,q}$ are independent of $\tau$ and $N$ with $N\tau\le T$.
\end{theorem}

The proof of Theorem \ref{thm:main} is based on discrete maximal parabolic regularity and will be presented in Section~\ref{sec:proof-main-thm}. 
For simplicity, we carry out the proof for the special case 
$f(\nabla u,u)=f(\nabla u)$. 
The proof for the general case is similar but contains additional 
lower order terms, which do not pose substantial difficulties in the analysis but clutter the formulas. 

Using Theorem~\ref{thm:main} together with energy estimates, the order of approximation can be improved to $k+1$ in the energy norm provided that the Runge-Kutta method satisfies the following two extra conditions: 

- The method is {\it algebraically stable}, that is, 
\begin{equation}
\begin{aligned}\label{alg-stab}
&\qquad\text{the weights $b_i$ are all positive and} \\[-1mm]
&\qquad\text{the $s\times s$ matrix with entries $b_ia_{ij}+b_ja_{ji}-b_ib_j$ is positive semidefinite.}
\end{aligned}
\end{equation}

- The quadrature formula with weights $b_i$ and nodes $c_i$  has at least order $k+1$:
\begin{equation} \label{order}
\sum_{i=1}^s b_i c_i^{l-1} = \frac 1l, \qquad\quad l=1,\dots,k+1.
\end{equation}
This is satisfied for the Radau IIA methods with $s\ge 2$ stages, for which the equations in \eqref{order} hold for $l\le 2s-1$ and which are algebraically stable; see  \cite[Section IV.12]{HW}.
We will prove the following result in Section~\ref{sec:energy}.

\begin{theorem}\label{thm:main-2} Consider a Runge--Kutta method of stage order $k$ that satisfies \eqref{rk-cond}, \eqref{alg-stab} and \eqref{order}, such as the Radau IIA method with $s=k\ge 2$ stages.
Assuming that the solution $u$ of \eqref{ivp} is sufficiently regular,  i.e., satisfies \eqref{Regularity} and
\begin{align}\label{Regularity+}
u\in H^{k+1}\big (0,T;H^1_0(\varOmega)\big ) 
\cap H^{k+2}\big (0,T;H^{-1}(\varOmega)\big )  ,
\end{align}
there exists a positive constant $\tau_2$ such that for
$\tau<\tau_2$ the discrete problem \eqref{RK} admits a unique solution 
that satisfies 
%
\begin{align}
\label{est-en} 
\max_{1\leq n\leq N}
\| u_{n}-u( t_{n})\|_{L^2(\varOmega)}
+
\bigg(\sum_{n=0}^N 
\tau \|\nabla\vec{u}_{n}- \nabla u(\vec t_{n})\|^2_{L^2(\varOmega)^{ds}} 
\bigg)^{\frac12}
\le C_2\tau^{k+1} .
\end{align}
The constant $C_2$ is independent of $\tau$ and $N$ with $N\tau\le T$.
\end{theorem}

\section{Auxiliary results related to maximal $L^p$ regularity}
\label{sec:maxreg}

The key to the error bounds of Theorems 2.1 and 2.2 is to control the $W^{1,\infty}(\varOmega)$ norm of the numerical solution. In this paper, this is done using the space-time Sobolev inequality, for $2/p+d/q < 1$,
\begin{equation} \label{sob-ineq}
\|v\|_{L^\infty(0,T;W^{1,\infty}(\varOmega))}
\le c_{p,q}
(\|\partial_tv\|_{L^p(0,T;L^q(\varOmega))}
+\|v\|_{L^p(0,T;W^{2,q}(\varOmega))})
\end{equation} 
together with the observation that the norm on the right-hand side is what is controlled by
{\it maximal $L^p$ regularity} for the solution of a linear parabolic problem with a second-order elliptic differential operator. Maximal $L^p$ regularity is characterized by Weis \cite{Weis01} in terms of the {\it $R$-boundedness} of the resolvent on a sector, a property that also yields {\it discrete maximal $\ell^p$-regularity} for the Runge--Kutta time discretization uniformly in the stepsize \cite{KLL}. In this section we present some results from this range of ideas and techniques. These results follow by suitably combining various results scattered in the literature. They will be important in the proof of Theorem~\ref{thm:main} and are also of independent interest.

\subsection{A Sobolev embedding}

\begin{lemma}\label{lem:sobolev}
If $2/p+d/q<1$, then there is the compact embedding
$$
W^{1,p}(0,T;L^q(\varOmega))\cap L^p(0,T;W^{2,q}(\varOmega)) \hookrightarrow 
C([0,T];W^{1,\infty}(\varOmega)).
$$
\end{lemma}
\noindent
This implies the bound \eqref{sob-ineq} for all $v\in W^{1,p}(0,T;L^q(\varOmega))\cap L^p(0,T;W^{2,q}(\varOmega))$.

\begin{proof}$\quad$
Via Sobolev embedding, we have
\begin{equation}\label{Imbedding}
\begin{aligned}
&W^{1,p}(0,T;L^q(\varOmega))\cap L^p(0,T;W^{2,q}(\varOmega))  \\
&\hookrightarrow 
L^\infty(0,T;(L^q(\varOmega),W^{2,q}(\varOmega))_{1-1/p,p})
\qquad\quad\,\, \text{see \cite[Proposition 1.2.10]{Lunardi95}}\\
&=
L^\infty(0,T;B^{2-2/p;q,p}(\varOmega)) 
\quad\,\,\, \text{by the definition of Besov spaces 
\cite[\textsection 7.32]{Adams}}. 
\end{aligned}
\end{equation}
Hence, 
$W^{1,p}(0,T;L^q(\varOmega))\cap L^p(0,T;W^{2,q}(\varOmega))$ 
is continuously embedded into the following space:
\begin{equation}\label{Imbedding-X}
X:=\{u\in L^\infty(0,T;B^{2-2/p;q,p}(\varOmega)) :\,
\partial_tu\in L^p(0,T;L^q(\varOmega))\} .
\end{equation}
If $2/p+d/q<1$, then there exists a small $\epsilon>0$ such that 
$2/p+\epsilon+d/q<1$, and so \cite[Theorem 7.34]{Adams} 
implies that $B^{2-2/p-\epsilon;q,p}(\varOmega)$ is continuously embedded into 
$W^{1,\infty}(\varOmega)$. 
Since $W^{2,q}(\varOmega)$ is compactly embedded into 
$W^{1,\infty}(\varOmega)$ (cf. \cite[Theorem 6.3]{Adams}) 
and 
$$B^{2-2/p;q,p}(\varOmega)=
(B^{2-2/p-\epsilon;q,p}(\varOmega),W^{2;q}(\varOmega))_{\theta,p},
\quad\mbox{with}
\,\,\,
\theta=\frac{\epsilon}{2/p+\epsilon} ,
$$
the Lions--Peetre theorem (\cite[Chapter V, Theorem 2.2]{LP}, see also \cite{Cobos}) implies that $B^{2-2/p;q,p}(\varOmega)$ is also compactly embedded into  
$W^{1,\infty}(\varOmega)$. 

Since $B^{2-2/p;q,p}(\varOmega)$ is compactly embedded into 
$W^{1,\infty}(\varOmega)$ and $W^{1,\infty}(\varOmega)$
is continuously embedded into $L^q(\varOmega)$, 
the Aubin--Lions--Simon lemma \cite[Theorem II.5.16]{BF13} 
implies that  $X$ is compactly embedded into 
$C([0,T];W^{1,\infty}(\varOmega))$. 
\QED\end{proof}

\subsection{An $R$-boundedness result}

We begin by recalling the notion of $R$-boundedness on $L^q$-spaces; see \cite{KW}.
A collection $\mathcal{T}$ of operators on $L^q(\varOmega)$
is {\it $R$-bounded} if and only if there is a constant $C_R$, called an $R$-bound of $\mathcal{T}$, such that
any finite subcollection of operators
$T_1,\dots,T_l\in \mathcal{T}$
satisfies
\begin{align*}
\bigg\|\bigg(\sum_{j=1}^l
|T_jv_j|^2\bigg)^{\frac{1}{2}}\bigg\|_{L^q(\varOmega)}
\! \! \! \leq C_R\bigg\|\bigg(\sum_{j=1}^l
|v_j|^2\bigg)^{\frac{1}{2}}\bigg\|_{L^q(\varOmega)} ,
\quad
\forall\,\, v_1,v_2,...,v_l\in L^q(\varOmega) .
\end{align*}
We will need the following result.

\begin{lemma}\label{lem:R-bound}
Let the elliptic operator 
$A :W^{2,q}(\varOmega)\cap W^{1,q}_0(\varOmega)\rightarrow 
L^q(\varOmega)$ with $1< q<\infty$ be defined by 
\begin{equation}
A\varphi = \sum_{i,j=1}^d\alpha_{ij}\partial_{i}\partial_j\varphi  ,
\end{equation} 
where the coefficient functions $\alpha_{ij}:\varOmega\to\R$ $(i,j=1,\dots,d)$ (which can be assumed symmetric: $\alpha_{ij}=\alpha_{ji}$) satisfy 
the following assumptions for some positive constants $\mu$ and $K$ and $\kappa$: 

(A1)  The coefficients are bounded in a H\"older norm:
$$
\|\alpha_{ij}\|_{C^\mu(\overline\varOmega)}\le K \, ;
$$

(A2) The symmetric coefficient matrix $(\alpha_{ij})$ 
satisfies the uniform ellipticity condition 
\begin{align} 
 \sum_{i,j=1}^d\alpha_{ij}(x)\xi_i \xi_j \ge
\kappa \sum_{j=1}^d \xi_j^2
\qquad
\forall\, x\in\varOmega,\,\,\forall\, \xi=(\xi_j)\in\C^d  \, .
\end{align}
%
%
Then, the collection of operators $\{z(z-A)^{-1}\,:\, |\arg z|<\theta\}$ 
is $R$-bounded on $L^q(\varOmega)$ for some $\theta\in(\pi/2,\pi)$. 
Both the $R$-bound and the angle $\theta$ depends only on $\mu$, $K$, $\kappa$, $\varOmega$ and $q$. 
\end{lemma}

\begin{proof}$\quad$
We argue by compactnesss.
Fix $K$, $\mu$, $\kappa$, $q\in(d,\infty)$, and an angle $\theta\in(\pi/2,\pi)$.
We denote by $M$ the set of all symmetric coefficient matrices $(\alpha_{ij})$ on $\varOmega$ satisfying conditions  (A1) and (A2). Clearly, $M$ is convex and closed in $\| \cdot\|_{C^\mu(\overline{\varOmega})}$ but also in the $\sup$-norm on $\overline{\varOmega}$. By the Arzela-Ascoli theorem, $M$ is compact in $\sup$-norm.

For any coefficient matrix $(\alpha_{ij})\in M$, the corresponding operator $A$ generates an analytic semigroup by \cite[Subsection 3.1.1]{Lunardi95}. This semigroup is positive, so 
$\max \mbox{Re}\,\sigma(A)$ is an eigenvalue. By \cite[Theorem 9.15]{GT} the half line $[0,\infty)$ belongs to the resolvent set of $A$. Thus $A$ is invertible and generates a bounded analytic semigroup.  
Moreover, for some $\theta_A\in(\pi/2,\pi)$, the set $\{z(z-A)^{-1}:|\arg z|<\theta_A\}$ is $R$-bounded with $R$-bound $R(A)$ (see \cite{PS}, \cite[Theorem 1.1]{K} or \cite[7.18]{KW}). 

If  $(\widetilde{\alpha}_{ij})\in M$ is another coefficient matrix with corresponding operator $\widetilde{A}$,
then
\begin{eqnarray*}
 \|(\widetilde{A}-A)u\|_{L^q(\varOmega)}
 &\le&\max_{ij}\|\widetilde{\alpha}_{ij}-\alpha_{ij}\|_{L^\infty(\varOmega)}
    \|u\|_{W^{2,q}(\varOmega)}\\ 
&\le& C_A\max_{ij}\|\widetilde{\alpha}_{ij}-\alpha_{ij}\|_{L^\infty(\varOmega)}\|Au\|_{L^q(\varOmega)}
\end{eqnarray*}
since $A$ is invertible and $D(A)=D(\widetilde{A})=W^{2,q}(\varOmega)\cap W^{1,q}_0(\varOmega)$, note that $C_A$ depends on $A$.
By the perturbation theorem for $R$-sectorial operators (\cite[Theorem 6.5]{KW}) we find $\eta_A>0$ such that $(\widetilde{\alpha}_{ij})\in M$, $\|\widetilde{\alpha}_{ij}-\alpha_{ij}\|_\infty<\eta_A$ implies
that for the operator $\widetilde{A}$ corresponding to $(\widetilde{\alpha}_{ij})$ the set $\{z(z-\widetilde{A})^{-1}:|\arg z|<\theta_A\}$ is $R$-bounded with $R$-bound $\le 2R(A)$.
By compactness of $M$ we thus find finitely many matrices $(\alpha^l_{ij})$ with corresponding operators $A_l$, $l\in F$, such that for each coefficient matrix $(\alpha_{ij})\in M$ with corresponding operator $A$ there is $l\in F$ with $\|\alpha_{ij}-\alpha_{ij}^l\|_\infty<\eta_{A_l}$. We conclude that, for $\theta:=\min_{l\in F}\theta_{A_l}$, the set $\{z(z-A)^{-1}:|\arg|<\theta\}$ is $R$-bounded with $R$-bound $\le 2\max_{l\in F}R(A_l)$.  
\QED\end{proof}

\subsection{Maximal $L^p$ regularity}

\begin{lemma} \label{lem:maxreg}
Under the conditions of Lemma~\ref{lem:R-bound}, the operator $A$ has
maximal $L^p$ regularity for $1<p<\infty$: for every $f\in L^p(0,T;L^q(\varOmega))$ (with arbitrary $T>0$),
the solution $u$ of the linear parabolic problem 
\begin{equation} \label{lin-par}
\begin{aligned}
\left\{\begin{array}{ll}
\displaystyle 
\frac{\partial u}{\partial t}-A u=f 
&\mbox{in}\,\,\,\varOmega,\\[10pt]
u=0&\mbox{on}\,\,\,\partial\varOmega ,
\end{array}\right.
\end{aligned}
\end{equation}
with zero initial values is bounded by
\begin{align}
\label{maxreg}
       \|\partial_tu\|_{L^p(0,T;L^q(\varOmega))} 
       + \|u\|_{L^p(0,T;W^{2,q}(\varOmega))}  
       \leq C_{p,q} \|f\|_{L^p(0,T;L^q(\varOmega))} ,
\end{align}
where the constant $C_{p,q}$ depends only on $K$, $\kappa$, $\varOmega$ and $p$ and $q$.
\end{lemma}

\begin{proof}$\quad$ 
By Lemma~\ref{lem:R-bound}, the operator-valued Mikhlin multiplier theorem used in Weis' characterization of maximal $L^p$-regularity \cite[Theorem 4.2]{Weis01} yields the maximal $L^p$ regularity
\begin{align*}
       \|\partial_tu\|_{L^p(0,T;L^q(\varOmega))} 
       + \|Au\|_{L^p(0,T;L^q(\varOmega))}  
       \leq C_{p,q} \|f\|_{L^p(0,T;L^q(\varOmega))} ,
\end{align*}
where $C_{p,q}$ depends only on $p,q$ and the $R$-bound of Lemma~\ref{lem:R-bound}.

Since $\alpha_{ij}\in W^{1,q}(\varOmega)\hookrightarrow C^{\alpha}(\overline\varOmega)$ for some $\alpha\in(0,1)$, 
\cite[Theorem 6.1 of Chapter 3]{CW98} implies that 
the elliptic operator 
$A: W^{2,q}(\varOmega)\cap W^{1,q}_0(\varOmega) 
\rightarrow L^q(\varOmega)$ is invertible and 
\begin{align} \label{W2q-est}
\|u\|_{W^{2,q}(\varOmega)}
\le C_q\|Au\|_{L^{q}(\varOmega)} ,
\end{align}
where $C_q$ depends only on $K$, $\kappa$, $\varOmega$ and $q$.
This yields the result.
\QED\end{proof}

\subsection{Discrete maximal $\ell^p$ regularity for Runge--Kutta methods}

As is shown in \cite[Theorem 5.1]{KLL},  A-stable Runge--Kutta methods with an invertible coefficient matrix preserve maximal $L^p$ regularity, uniformly in the stepsize. Before we formulate the Runge--Kutta analog of Lemma~\ref{lem:maxreg}, we need to introduce further notation.

For any Banach space $X$ and any sequence 
$(v_n)_{n=1}^N$ with entries in $X$ we denote, for a given stepsize $\tau>0$, 
\[ \big\|(v_n)_{n=1}^N\big\|_{L^p(X)} 
:= \Bigl( \sum_{n=1}^N\tau  \| v_n\|_X^p \Bigr)^{1/p},\]
which is the $L^p(0,N\tau;X)$ norm of the 
piecewise constant function that equals $v_n$ on the time interval 
$(t_{n-1},t_n]$. We use the same notation also for sequences $(v_n)_{n=0}^N$, replacing $n=1$ by $n=0$ in the sum.

Considering the piecewise linear interpolant of a sequence  $(v_n)_{n=1}^N$ in $W^{2,q}(\varOmega)$ and the starting value $v_0=0$, 
Lemma~\ref{lem:sobolev} gives, 
for $2/p+d/q<1$,  
\begin{equation}\label{DImbedding} 
\begin{aligned}
&\|(v_n)_{n=1}^{N}\|_{L^\infty(W^{1,\infty}(\varOmega))}  \\
&\le c_{p,q} \left(
\bigg\|\bigg(\frac{v_n-v_{n-1}}{\tau}\bigg)_{n=1}^N
\bigg\|_{L^p(L^q(\varOmega))}
+\|(v_n)_{n=1}^N \|_{L^p(W^{2,q}(\varOmega))}\right) .
\end{aligned}
\end{equation} 

We now consider the Runge--Kutte time discretization of the linear parabolic problem \eqref{lin-par} with stepsize $\tau$,
\begin{equation}\label{lin-rk}
u_{n,i} = u_n + \tau \sum_{j=1}^s a_{ij} \bigl(Au_{n,j} + f_{n,j} \bigr)
\qquad (i=1,\dots,s),
\end{equation}
and $u_{n+1}=u_{n,s}$ for a Runge--Kutta method with \eqref{rk-cond}. We use again the vector notation of \eqref{vec-not}, $\vec u_n=(u_{n,i})_{i=1}^s$ and $\vec f_n=(f_{n,i})_{i=1}^s$. We then have the following time-discrete analog of Lemma~\ref{lem:maxreg}.

\begin{lemma} \label{lem:rk-maxreg}
Consider a Runge--Kutta method that satisfies \eqref{rk-cond}, such as the $s$-stage Radau IIA method. Under the conditions of Lemma~\ref{lem:R-bound}, there is discrete maximal $L^p$ regularity for $1<p<\infty$ uniformly in the stepsize $\tau>0$: for every sequence $(\vec f_n)_{n=0}^N$ with entries in $L^q(\varOmega)^s$ (with arbitrary $N\ge 1$), the numerical solution defined by \eqref{lin-rk} with zero initial value $u_0=0$ satisfies the bound, with $\vec u_{-1}=0$,
\begin{align}\label{un-un-1LpW2qD}
        \bigg\|\bigg(\frac{\vec u_n-\vec u_{n-1}}{\tau}\bigg)_{n=0}^N\bigg\|_{L^p(L^q(\varOmega)^s)} 
       +  \big\|(\vec u_n)_{n=0}^N\big\|_{L^p(W^{2,q}(\varOmega)^s)} \nonumber\\
       \leq C_{p,q}  \big\|(\vec f_{n})_{n=0}^N\big\|_{L^p(L^q(\varOmega)^s)},
\end{align}
where  the constant $C_{p,q}$ depends only on $K$, $\kappa$, $\varOmega$ and $p$ and $q$. In particular, $C_{p,q}$ is independent of $N$ and $\tau$. 
\end{lemma}

\begin{proof}$\quad$ In view of Lemma~\ref{lem:R-bound}, \cite[Theorem 5.1]{KLL}  gives the bound,
with $\vec{\dot u}_n = (\dot u_{n,j})_{j=1}^s$  for $\dot u_{n,j} = Au_{n,j}+f_{n,j}$,
$$
\| (\vec{\dot u}_n)_{n=0}^N \|_{L^p(L^q(\varOmega)^s)} + \| (A\vec{u_n})_{n=0}^N \|_{L^p(L^q(\varOmega)^s)} \le \widetilde C_{p,q}  \big\|(\vec f_{n})_{n=0}^N\big\|_{L^p(L^q(\varOmega)^s)},
$$
where $\widetilde C_{p,q}$ depends only on $p,q$ and the $R$-bound of Lemma~\ref{lem:R-bound}.

For the second term on the left-hand side we recall \eqref{W2q-est}. For the first term we note that \eqref{lin-rk} yields
$$
\biggl\| \biggl( \frac{u_{n,i}-u_n}\tau \biggr)_{i=1}^s \biggr\|_{L^q(\varOmega)^s} \le
\gamma \, \bigl\| \bigl( \dot u_{n,j}  \bigr)_{j=1}^s \bigr\|_{L^q(\varOmega)^s},
$$
where $\gamma$ is the norm of the Runge--Kutta coefficient matrix $(a_{ij})$. Writing
$$
u_{n,i}-u_{n-1,i} = (u_{n,i}-u_n) + (u_n-u_{n-1}) - (u_{n-1,i}-u_{n-1})
$$
and noting that $u_n-u_{n-1}=u_{n-1,s}-u_{n-1}$, we find that the above inequality (for $n$ and $n-1$) yields
$$
 \bigg\|\frac{\vec u_n-\vec u_{n-1}}{\tau}\bigg\|_{L^q(\varOmega)^s} 
 \le \gamma \, \bigl\| \bigl( \dot u_{n,j}  \bigr)_{j=1}^s \bigr\|_{L^q(\varOmega)^s} + 2
 \gamma \, \bigl\| \bigl( \dot u_{n-1,j}  \bigr)_{j=1}^s \bigr\|_{L^q(\varOmega)^s},
 $$
 which completes the proof of the result.
 \QED\end{proof}
 
 Combining Lemma~\ref{lem:rk-maxreg} and  \eqref{DImbedding}, we thus obtain the bound
\begin{equation}\label{rk-w18}
 \big\|(\vec u_n)_{n=0}^N\big\|_{L^\infty(W^{1,\infty}(\varOmega)^s)} 
       \leq \widehat C_{p,q}  \big\|(\vec f_{n})_{n=0}^N\big\|_{L^p(L^q(\varOmega)^s)},
\end{equation}
 with $\widehat C_{p,q}=c_{p,q}C_{p,q}$.
 This $W^{1,\infty}$ bound of the numerical solution is the key to proving Theorem~\ref{thm:main}.

\subsection{Nonautonomous linear parabolic problems}
Let the time-dependent elliptic operators $A(t) :W^{2,q}(\varOmega)\cap W^{1,q}_0(\varOmega)\rightarrow 
L^q(\varOmega)$ for $0\le t \le T$ be defined by 
\begin{equation}
A(t)\varphi = \sum_{i,j=1}^d\alpha_{ij}(\cdot,t)\partial_{i}\partial_j\varphi  ,
\end{equation} 
where the coefficient functions $\alpha_{ij}(\cdot,t):\varOmega\to\R$ $(i,j=1,\dots,d)$ satisfy conditions (A1) and (A2) of Lemma~\ref{lem:R-bound} uniformly for $0\le t \le T$ and additionally the Lipschitz condition
\begin{equation}
 \label{Lip-t}
 \| \alpha_{ij}(\cdot,t) - \alpha_{ij}(\cdot,s) \|_{L^\infty(\varOmega)} \le L\, |t-s|,
 \qquad 0\le s,t\le T.
\end{equation}

\begin{lemma}
 \label{lem:maxreg-t}
 In the above situation of time-dependent elliptic operators $A(t)$, the solution of the nonautonomous linear problem \eqref{lin-par} is bounded by \eqref{maxreg}, where the constant $C_{p,q}$ depends additionally on $L$ and $T$.
\end{lemma}

\begin{proof}$\quad$
 For $0\le t \le \bar t \le T$, we rewrite the differential equation as
 $$
 \partial_t u(t) = A(\bar t)u(t) - \bigl( A(\bar t) - A(t) \bigr)u(t) + f(t)
 $$
 and apply Lemma~\ref{lem:maxreg} for the operator $A(\bar t)$ to bound
 \begin{align} \label{maxreg-prov}
     \|\partial_tu\|_{L^p(0,\bar t;L^q(\varOmega))} 
       + \|u\|_{L^p(0,\bar t;W^{2,q}(\varOmega))}  
      & \leq C_{p,q} \|  \bigl( A(\bar t) - A(\cdot) \bigr)u \|_{L^p(0,\bar t;L^q(\varOmega))} 
     \\ &\quad +  C_{p,q} \|f\|_{L^p(0,\bar t;L^q(\varOmega))} . \nonumber
 \end{align}
We denote
$$
\eta(\bar t)= \| u \|_{L^p(0,\bar t;W^{2,q}(\varOmega))} ^p.
$$
By the Lipschitz condition \eqref{Lip-t} and by partial integration we obtain
\begin{align*}
\int_0^{\bar t}  \|\bigl( A(\bar t) - A(t) \bigr)u(t)\|_{L^q(\varOmega)}^p \,\d t &\le 
L^p \int_0^{\bar t} (\bar t -t)^p \| u(t) \|_{W^{2,q}(\varOmega)}^p \,\d t
\\ &= L^{p}p\int_0^{\bar t} (\overline t - t)^{p-1}\eta(t)\,\d t.
\end{align*}
Hence we have from \eqref{maxreg-prov}
$$
\eta(\bar t) \le C \int_0^{\bar t} (\overline t - t)^{p-1}\eta(t)\,\d t + C \|f\|_{L^p(0,\bar t;L^q(\varOmega))}^p, \qquad 0\le \bar t \le T,
$$
and a Gronwall inequality yields
$$
\eta(T) \le C' \|f\|_{L^p(0,T;L^q(\varOmega))}^p,
$$
which combined with \eqref{maxreg-prov} yields the result.
\QED\end{proof}

\subsection{Runge--Kutta discretization of nonautonomous linear problems}
With Lemma~\ref{lem:rk-maxreg}, the previous result for the nonautonomous linear problem extends to its Runge--Kutta time discretization
\begin{equation}\label{lin-rk-t}
u_{n,i} = u_n + \tau \sum_{j=1}^s a_{ij} \bigl(A(t_{n,j})u_{n,j} + f_{n,j} \bigr)
\qquad (i=1,\dots,s),
\end{equation}
and $u_{n+1}=u_{n,s}$ for a Runge--Kutta method with \eqref{rk-cond}.

\begin{lemma} \label{lem:rk-maxreg-t}
Consider a Runge--Kutta method that satisfies \eqref{rk-cond}, such as the $s$-stage Radau IIA method. Under the conditions of Lemma~\ref{lem:maxreg-t}, there is discrete maximal $L^p$ regularity for $1<p<\infty$ uniformly in the stepsize $\tau>0$: for every sequence $(\vec f_n)_{n=0}^N$ with entries in $L^q(\varOmega)^s$ (with arbitrary $N\ge 1$), the numerical solution defined by \eqref{lin-rk-t} with zero initial value $u_0=0$ satisfies the bound \eqref{un-un-1LpW2qD}, where $C_{p,q}$ is independent of $N$ and $\tau$ with $N\tau\le T$, but depends on $T$. 
\end{lemma}

\begin{proof}$\quad$
 The result follows from Lemma~\ref{lem:rk-maxreg} in the same way as Lemma~\ref{lem:maxreg-t} follows from Lemma~\ref{lem:maxreg}, using a partial summation in place of the partial integration.
\QED\end{proof}

\section{Proof of Theorem \ref{thm:main}}\label{sec:proof-main-thm}

\subsection{Defects and error equation}
The exact solution values satisfy the Runge--Kutta relations up to a defect:
$$
u(t_n+c_i\tau) = u(t_n) + \tau \sum_{j=1}^s a_{ij} \, \partial_tu(t_n+c_j\tau) + d_{n,i},
$$
where we note that $d_{n,i}$ is the quadrature error over the interval $[t_n,t_n+c_i\tau]$ of the quadrature formula with weights $a_{ij}$ and nodes $c_j$. Using Taylor expansion at $t_n$ and the definition of the stage order \eqref{stage-order} and the regularity condition \eqref{Regularity}, we can bound  $\vec d_n =(d_{n,i})_{i=1}^s$ by
$$
\| (\vec d_n )_{n=0}^N \|_{L^p(L^q(\varOmega)^s)} \le C\tau^{k+1}.
$$
We rewrite the above equation as
\begin{align}\label{ExactEq}
u(t_n+c_i\tau) = u(t_n) + \tau \sum_{j=1}^s a_{ij} \, (\partial_tu(t_n+c_j\tau) -r_{n,j}),
\end{align}
where $\vec r_n = (r_{n,j})_{j=1}^s$ is the solution of the linear system with the invertible Runge--Kutta matrix $(a_{ij})$,
\begin{equation}\label{r-bound}
\tau \sum_{j=1}^s a_{ij} r_{n,j} = -d_{n,i}, \qquad\text{so that}\quad \rho:= \| (\vec r_n )_{n=0}^N \|_{L^p(L^q(\varOmega)^s)} \le C\tau^{k}.
\end{equation}
We rewrite the partial differential equation as
\begin{align}\label{ExDerivative}
\partial_t u = \nabla \cdot f(\nabla u) =  \sum_{k,l=1}^d f_{k,l}(\nabla u) \partial_k\partial_l u,
\quad\text{ with }\ f_{k,l}=\partial f_k/\partial p_l.
\end{align}
Comparing \eqref{RK} and \eqref{RKDerivative} with 
\eqref{ExactEq} and \eqref{ExDerivative}, we see that
the errors 
\begin{align}\label{Error1}
e_{n,i}:= u_{n,i}-u(t_{n,i})
\quad\mbox{and}\quad
\dot e_{n,j}:= \dot u_{n,j} - \partial_tu(t_{n,j}) +r_{n,j}
\end{align} 
 satisfy the error equations (for $i,j=1,\dots,s$)
\begin{subequations}\label{err-eq}
\begin{align}
\label{err-eq-1}
 & e_{n,i} = e_n + \tau \sum_{j=1}^s a_{ij} \dot e_{n,j} ,\qquad\quad e_{n+1}=e_{n,s}
 \\
 \nonumber
 & \dot e_{n,j} = \sum_{k,l=1}^d f_{k,l}(\nabla u(t_{n,j})) \partial_k\partial_l e_{n,j} 
 \\ & +
 \sum_{k,l=1}^d \Bigl( f_{k,l}(\nabla u(t_{n,j}) + \nabla e_{n,j}) - f_{k,l}(\nabla u(t_{n,j})) \Bigr)
\partial_k\partial_l   \bigl( u(t_{n,j}) + e_{n,j} \bigr) +r_{n,j}.
\label{err-eq-2}
\end{align}
\end{subequations}
Clearly, $\vec e_n=(e_{n,i})$ is a solution of the error equations \eqref{err-eq} if and only if $(u_{n,i})=(u(t_{n,i})+e_{n,i})$ is a solution of the Runge--Kutta equations \eqref{last-stage}-\eqref{RK}.

\subsection{Error bound}

We first show the error bound of Theorem~\ref{thm:main} under the additional assumption that the errors remain bounded by a small constant in the $W^{1,\infty}$ norm. This condition will be verified in the next subsection.

\begin{lemma} \label{lem:error-bound} In the situation of Theorem~\ref{thm:main}, suppose that the error equations have a solution $(e_{n,i})$  for $0\le n \le N$ and $i=1,\dots,s$ such that
$$
\max_{0\le n \le N} \max_{1\le i \le s}\| e_{n,i} \|_{W^{1,\infty}(\varOmega)} \le \mu
$$
with a sufficiently small constant $\mu$ (independent of $\tau$ and $N$ with $N\tau\le T$).
Then the $O(\tau^k)$ error bounds \eqref{err-est} are satisfied.
\end{lemma}

\begin{proof}$\quad$ If we consider $\vec g_n =(g_{n,j})$ with
$$
g_{n,j} =  \sum_{k,l=1}^d \Bigl( f_{k,l}(\nabla u(t_{n,j}) + \nabla e_{n,j}) - f_{k,l}(\nabla u(t_{n,j})) \Bigr)
\partial_k\partial_l   \bigl( u(t_{n,j}) + e_{n,j} \bigr)
$$
as an inhomogeneity in \eqref{err-eq-2}, then Lemma~\ref{lem:rk-maxreg-t} shows that
\begin{align} \label{e-bound}
&  \bigg\|\bigg(\frac{\vec e_n-\vec e_{n-1}}{\tau}\bigg)_{n=0}^N\bigg\|_{L^p(L^q(\varOmega)^s)} 
       +  \big\|(\vec e_n)_{n=0}^N\big\|_{L^p(W^{2,q}(\varOmega)^s)} 
       \\
      & \qquad\qquad\qquad\leq C  \Bigl( \big\|(\vec g_{n})_{n=0}^N\big\|_{L^p(L^q(\varOmega)^s)}
      + \big\|(\vec r_{n})_{n=0}^N\big\|_{L^p(L^q(\varOmega)^s)} \Bigr).
      \nonumber
\end{align}
We bound, with a local Lipschitz constant $L$ of $f_{k,l}$,
\begin{equation}\label{g-bound}
\| g_{n,j} \|_{L^q(\varOmega)} \le 
L \| \nabla e_{n,j} \|_{L^\infty(\varOmega)}\,\| u(t_{n,j}) \|_{W^{2,q}(\varOmega)} + 
L \|  \nabla e_{n,j} \|_{L^\infty(\varOmega)}\,\| e_{n,j} \|_{W^{2,q}(\varOmega)}
\end{equation}
so that
$$
\big\|(\vec g_{n})_{n=0}^N\big\|_{L^p(L^q(\varOmega)^s)} \le 
C_1 \big\|(\vec e_{n})_{n=0}^N\big\|_{L^p(W^{1,\infty}(\varOmega)^s)} + 
C_2\mu  \big\|(\vec e_n)_{n=0}^N\big\|_{L^p(W^{2,q}(\varOmega)^s)}.
$$
Using the bound
$$
\| \vec e_{n} \|_{W^{1,\infty}(\varOmega)^s} \le \mu \| \vec e_{n} \|_{W^{2,q}(\varOmega)^s} +
C_\mu \| \vec e_{n} \|_{L^{q}(\varOmega)^s},
$$ 
we obtain
$$
\big\|(\vec g_{n})_{n=0}^N\big\|_{L^p(L^q(\varOmega)^s)} \le 
C\mu  \big\|(\vec e_n)_{n=0}^N\big\|_{L^p(W^{2,q}(\varOmega)^s)} + 
C_\mu\big\|(\vec e_n)_{n=0}^N\big\|_{L^p(L^{q}(\varOmega)^s)}.
$$
If $\mu$ is sufficiently small, then the first term on the right-hand side can be absorbed in the left-hand side of \eqref{e-bound}, and we are left with 
\begin{align*}
&  \bigg\|\bigg(\frac{\vec e_n-\vec e_{n-1}}{\tau}\bigg)_{n=0}^N\bigg\|_{L^p(L^q(\varOmega)^s)} 
       +  \big\|(\vec e_n)_{n=0}^N\big\|_{L^p(W^{2,q}(\varOmega)^s)} 
       \\
      & \qquad\qquad\qquad\leq C  \Bigl( \big\|(\vec e_{n})_{n=0}^N\big\|_{L^p(L^q(\varOmega)^s)}
      + \big\|(\vec r_{n})_{n=0}^N\big\|_{L^p(L^q(\varOmega)^s)} \Bigr).
\end{align*}
Such a bound holds not only for the final $N$, but for each $\bar n \le N$. We write
$\vec e_n =
\tau\sum_{m=0}^n(\vec e_m-\vec e_{m-1})/\tau$ and use, for $\alpha_j=\frac1\tau\| \vec e_j-\vec e_{j-1}\|_{L^q(\varOmega)^s}$, the inequality
\begin{equation}\label{triangle-inequality}
 \biggl\| \biggl( \sum_{j=0}^m \alpha_j \biggr)_{m=0}^{\bar n} \biggr\|_p \le \sum_{m=0}^{\bar n} \left\| \left(\alpha_j\right)_{j=0}^m \right\|_p\,,
\end{equation}
which is just the triangle inequality for the sum of vectors in $\R^{\bar n+1}$
$$
\begin{pmatrix} 0 \\ \vdots \\ 0 \\ \alpha_0 \end{pmatrix} +
\begin{pmatrix} 0 \\ \vdots  \\ \alpha_0 \\ \alpha_1 \end{pmatrix} + \ldots +
\begin{pmatrix} 0 \\ \alpha_0 \\ \vdots  \\ \alpha_{\bar n-1} \end{pmatrix} +
\begin{pmatrix} \alpha_0 \\ \alpha_1 \\ \vdots  \\ \alpha_{\bar n} \end{pmatrix}.
$$
We thus obtain, for $0\le \bar n \le N$,
\begin{align*}
&  \bigg\|\bigg(\frac{\vec e_n-\vec e_{n-1}}{\tau}\bigg)_{n=0}^{\bar n}\bigg\|_{L^p(L^q(\varOmega)^s)} 
       +  \big\|(\vec e_n)_{n=0}^{\bar n}\big\|_{L^p(W^{2,q}(\varOmega)^s)} 
       \\
      & \qquad\qquad\qquad\leq C  \biggl( 
      \tau\sum_{m=0}^{\bar n} \bigg\|\bigg(\frac{\vec e_n-\vec e_{n-1}}{\tau}\bigg)_{n=0}^{m} \bigg\|_{L^p(L^q(\varOmega)^s)} 
      + \big\|(\vec r_{n})_{n=0}^{\bar n}\big\|_{L^p(L^q(\varOmega)^s)} \biggr).
\end{align*}
Applying a discrete Gronwall inequality then yields
\begin{align}\label{e-est}
  \bigg\|\bigg(\frac{\vec e_n-\vec e_{n-1}}{\tau}\bigg)_{n=0}^N\bigg\|_{L^p(L^q(\varOmega)^s)} 
       +  \big\|(\vec e_n)_{n=0}^N\big\|_{L^p(W^{2,q}(\varOmega)^s)} 
       \leq \widetilde C 
       \big\|(\vec r_{n})_{n=0}^N\big\|_{L^p(L^q(\varOmega)^s)} ,
\end{align}
and the result follows with the bound \eqref{r-bound}.
\QED\end{proof}

\subsection{Existence of the numerical solution}

In this subsection, we prove the existence of a solution 
$\vec e_{n}$ for \eqref{err-eq} satisfying the error bound \eqref{e-est} by using Schaefer's fixed point theorem via the arguments of the proof of Lemma~\ref{lem:error-bound}, which rely on the maximal regularity properties of Section~\ref{sec:maxreg}.

\begin{lemma}[Schaefer's fixed point theorem {\cite[Chapter 9.2, Theorem 4]{Evans}}]
\label{THMSchaefer}
Let $X$ be a Banach space and let ${\mathcal M}:X\rightarrow X$ be 
a continuous and compact map. If the set 
\begin{equation}
\bigl\{\phi\in X:\; \phi=\theta {\mathcal M}(\phi) \ \hbox{ for some }\, 
\theta\in [0,1] \bigr\}
\end{equation}
is bounded in $X$, then the map ${\mathcal M}$ has a fixed point.
\end{lemma}

We define a map ${\mathcal M}: 
C([0,T],W^{1,\infty}(\varOmega )^s)
\rightarrow C([0,T],W^{1,\infty}(\varOmega )^s)$ in the following way: 
for any given $\vec\varphi=(\varphi_{j})_{j=1}^{s}
\in C([0,T],W^{1,\infty}(\varOmega )^s)$, we define 
$\vec e:={\mathcal M} \vec\varphi$ 
as the piecewise linear interpolation in time of the vectors $\vec e_n=(e_{n,i})_{i=1}^s$ for $n=0,\dots,N$ (that is, interpolating linearly between $e_{n,i}$ and $e_{n-1,i}$ for each $i$), where  $\vec e_n=(e_{n,i})_{i=1}^s$ are the solution of the {\it linear} problem 
\begin{subequations}\label{err-eq-M}
\begin{align}
\label{err-eq-1-M}
 &e_{n,i} = e_n + \tau \sum_{j=1}^s a_{ij} \dot e_{n,j} ,\qquad\quad e_{n+1}=e_{n,s}
 \\
\label{err-eq-2-M}
& \dot e_{n,j} = \sum_{k,l=1}^d f_{k,l}\bigl(\nabla u(t_{n,j})\bigr) \partial_k\partial_l e_{n,j} 
 \\ & +
 \sum_{k,l=1}^d \Bigl(f_{k,l}\bigl(\nabla u(t_{n,j}) + \beta(\varphi_j(t_{n,j}))\nabla \varphi_j(t_{n,j})\bigr) - f_{k,l}\bigl(\nabla u(t_{n,j})\bigr) \Bigr) \times\nonumber\\[-3mm]
&\hspace{7.5cm}\partial_k\partial_l   \bigl( u(t_{n,j}) + e_{n,j} \bigr) \nonumber\\
&+r_{n,j},
\nonumber
\end{align}
\end{subequations}
where 
$$
\beta(\varphi) = \min\Bigl( \frac{\sqrt\rho}{\|\varphi\|_{W^{1,\infty}(\varOmega)}}, 1\Bigr),
$$
which has the following properties: 
\begin{subequations}
\begin{align} 
&\|\beta(\varphi) \varphi\|_{W^{1,\infty}(\varOmega )} 
\le \sqrt{\rho} ,\\
&\beta(\varphi)=1\quad\mbox{if}\quad  
\|\varphi\|_{W^{1,\infty}(\varOmega )}\le \sqrt{\rho} .
\end{align}
\end{subequations}

\begin{lemma} \label{lem:M}
The map ${\mathcal M}: 
C([0,T],W^{1,\infty}(\varOmega )^s)
\rightarrow C([0,T],W^{1,\infty}(\varOmega )^s)$
is well defined, continuous and compact. 
\end{lemma}

\begin{proof}$\quad$
 Following the lines of the proof of Lemma~\ref{lem:error-bound}, with the only difference that 
 $ \|\nabla e_{n,j}\|_{L^\infty(\varOmega)}$ is replaced with 
 $\|\beta(\varphi_j(t_{n,j}))\nabla \varphi_j(t_{n,j})\|_{L^\infty(\varOmega)}\le \sqrt\rho$ in \eqref{g-bound},
 it is seen that $\mathcal{M}$ maps boundedly into the space
 $W^{1,p}(0,T;L^q(\varOmega)^s) \cap L^p (0,T; W^{2,q}(\varOmega)^s)$, which is compactly embedded in
 $C([0,T],W^{1,\infty}(\varOmega )^s)$ by Lem\-ma~\ref{lem:sobolev}. The continuity of $\mathcal{M}$ is also obtained by the arguments used in the proof of
 Lemma~\ref{lem:error-bound}.
\QED\end{proof}

To apply Schaefer's fixed point theorem (Lemma \ref{THMSchaefer}), 
we assume that 
\begin{align*}
\vec \varphi 
=\theta {\mathcal M}\vec \varphi 
\quad\mbox{for some $\theta\in[0,1]$}.
\end{align*}
Then $\vec e:={\mathcal M}\vec \varphi $ 
is the piecewise linear interpolation of the solution of the equations \eqref{err-eq-M} with $\varphi_j=\theta e_{j}$. Using the same proof as that of Lemma~\ref{lem:error-bound}, it is now seen that $\vec e$ satisfies
$O(\rho)=O(\tau^k)$ error bounds \eqref{err-est}. This implies that $\|\vec \varphi\|_{W^{1,\infty}(\varOmega)^s}\le C\rho$ (note that then $\beta(\varphi_j)=1$), and hence Schaefer's fixed point theorem yields the existence of a solution to the error equations~\eqref{err-eq} satisfying \eqref{err-est}.

\subsection{Uniqueness of the numerical solution}

The stability result of Lemma~\ref{lem:error-bound}, that is, the bound \eqref{e-est} used with ${\vec r_n=0}$,
 implies the local uniqueness of the Runge-Kutta solution in an  $W^{1,\infty}(\varOmega)$ neighbourhood of width $\mu$ (sufficiently small but independent of the stepsize $\tau$).
 
\section{Proof of Theorem~\ref{thm:main-2}} \label{sec:energy} 
The proof is similar to previous proofs of error bounds for Runge--Kutta time discretizations of parabolic problems using energy estimates \cite{LO,DLM}. In particular, the same use is made of the algebraic stability condition \eqref{alg-stab}. However, the proof differs in that here we need to invoke the $W^{1,\infty}(\varOmega)$ error bounds provided by Theorem~\ref{thm:main}.

\subsection{Defects}
We denote the exact solution values  $u_{n,i}^*=u(t_n+c_i\tau)$, $\dot u_{n,i}^*=\partial_t u(t_n+c_i\tau)$, and $u_{n}^* =u(t_n)$. Note that $u_{n+1}^*=u_{n,s}^*$ by our condition $c_s=1$. We denote by $d_{n,i}$ and $d_{n+1}$ the defects obtained on inserting the exact solution into the Runge--Kutta equations,
$$
u_{n,i}^* = u_n^* + \tau \sum_{j=1}^s a_{ij} \dot u_{n,j}^* + d_{n,i}, \qquad
u_{n+1}^* = u_n^* + \tau \sum_{j=1}^s b_j \dot u_{n,j}^* + d_{n+1}.
$$
The defects are thus quadrature errors.
By Taylor expansion at $t_n$ and the definition of the stage order \eqref{stage-order} and by condition \eqref{order}, the defects are of the form
\begin{align*}
d_{n,i} &= \tau^k \int_{t_n}^{t_{n+1}} K_i\Bigl(\frac{t-t_n}\tau\Bigr) \, u^{(k+1)}(t)\,\d t \\
d_{n+1} &= \tau^{k+1} \int_{t_n}^{t_{n+1}} K\Bigl(\frac{t-t_n}\tau\Bigr) \, u^{(k+2)}(t)\,\d t \\
&=
- \tau^k \int_{t_n}^{t_{n+1}} K'\Bigl(\frac{t-t_n}\tau\Bigr) \, u^{(k+1)}(t)\,\d t
\end{align*}
with bounded Peano kernels $K_i$ and $K$. Here we assume for simplicity that all $c_i\in[0,1]$, as is the case for all methods of interest. In the following we denote  
by $\langle\cdot,\cdot\rangle$ the duality pairing between $H^1_0(\varOmega)$ and $H^{-1}(\varOmega)$, which restricted to
$L^2(\varOmega)\times L^2(\varOmega)$ coincides with the $L^2(\varOmega)$ inner product. We further denote
$$
|\cdot|=\|\cdot\|_{L^2(\varOmega)}, \quad \|\cdot \| = \|\cdot\|_{H^1_0(\varOmega)},
\quad \|\cdot \|_\star = \|\cdot\|_{H^{-1}(\varOmega)}.
$$
We define $\delta\ge 0$ by setting
\begin{equation}\label{delta}
\delta^2 = \tau\sum_{n=0}^N \sum_{i=1}^s \| d_{n,i} \|^2 + \tau \sum_{n=0}^N
\bigl( \| d_{n+1} \|^2 + \| d_{n+1}/\tau \|_\star^2 \bigr)
\end{equation}
and note that by our regularity assumption and the above defect estimates we have
$$
\delta \le C\tau^{k+1}.
$$

\subsection{Error equations}
The errors $e_{n,i}=u_{n,i}-u_{n,i}^*$, $\dot e_{n,i}=\dot u_{n,i}-\dot u_{n,i}^*$, and
$e_n=u_n-u_n^*$ satisfy the error equations (with $f_{k,l}=\partial f_k/\partial p_l$)
\begin{subequations}
\begin{align}
 \dot e_{n,i} &= \sum_{k,l=1}^d f_{k,l}(\nabla u_{n,i}^*) \partial_k\partial_l e_{n,i} +
\sum_{k,l=1}^d \bigl( f_{k,l}(\nabla u_{n,i}) - f_{k,l}(\nabla u_{n,i}^*) \bigr)  \partial_k\partial_lu_{n,i}
\label{err-a}
\\
e_{n,i} &= e_n + \tau \sum_{j=1}^s a_{ij} \dot e_{n,j} - d_{n,i}
\label{err-b}
\\
e_{n+1} &= e_{n} + \tau \sum_{i=1}^s b_{i} \dot e_{n,i} - d_{n+1}.
\label{err-c}
\end{align}
\end{subequations}

\subsection{Energy estimate using algebraic stability}
Taking the square of the $L^2(\varOmega)$ norm in \eqref{err-c} yields
\begin{equation}\label{err-square}
|e_{n+1}|^2 = \bigl| e_{n} + \tau \sum_{i=1}^s b_{i} \dot e_{n,i} \bigr|^2 -
2\bigl\langle d_{n+1}, e_{n} + \tau \sum_{i=1}^s b_{i} \dot e_{n,i} \bigr\rangle +
|d_{n+1}|^2.
\end{equation}
The three terms on the right-hand side will now be estimated separately. We express $e_n$ by \eqref{err-b} to obtain
\begin{align*}
\bigl| e_{n} + \tau \sum_{i=1}^s b_{i} \dot e_{n,i} \bigr|^2 &= |e_n|^2 +
2\tau \sum_{i=1}^s b_i \langle \dot e_{n,i}, e_{n,i}+d_{n,i} \rangle \\
& +\ \tau^2 \sum_{i=1}^s \sum_{j=1}^s (b_ib_j - b_i a_{ij} - b_j a_{ji})\, \langle \dot e_{n,i},\dot e_{n,j} \rangle.
\end{align*}
Here the last term is nonpositive by the algebraic stability condition \eqref{alg-stab}. We next estimate the second term on the right-hand side. Omitting momentarily all subscripts $n,i$ for clarity of notation, we have by \eqref{err-a} 
\begin{align} \label{eed}
\langle \dot e, e+d \rangle &= \Bigl\langle \sum_{k,l=1}^d f_{k,l}(\nabla u^*) \partial_k\partial_l e, e+d \Bigr\rangle \\
&+ \Bigl\langle \sum_{k,l=1}^d \bigl( f_{k,l}(\nabla u) - f_{k,l}(\nabla u^*) \bigr)  \partial_k\partial_l u, e+d \Bigr\rangle.
\nonumber
\end{align}
For the first term on the right-hand side we use partial integration to write
\begin{align*}
&\Bigl\langle \sum_{k,l=1}^d f_{k,l}(\nabla u^*) \partial_k\partial_l e, e+d \Bigr\rangle  =
-  \sum_{k,l=1}^d \Bigl\langle \partial_l e, \partial_k \bigl( f_{k,l}(\nabla u^*)(e+d) \bigr)\Bigr\rangle 
\\
&= - \sum_{k,l=1}^d \Bigl\langle  \partial_l e, f_{k,l}(\nabla u^*) (\partial_k e + \partial_k d)  \Bigr\rangle 
 - \sum_{k,l=1}^d\Bigl\langle  \partial_l e,\partial_k \bigl( f_{k,l}(\nabla u^*) \bigr) (e+d)  \Bigr\rangle
 \\
 &\equiv I_1 +I_2.
\end{align*}
Under the regularity condition \eqref{Regularity} about the exact solution we have a bound $\| u^* \|_{W^{1,\infty}(\varOmega)} \le R < \infty$, and hence there exists $\kappa_R>0$ such that we have for all $x\in\varOmega$
$$
\sum_{k,l=1}^d f_{k,l}(\nabla u^*(x)) \xi_k \xi_l \ge \kappa_R \sum_{l=1}^d \xi_l^2 \qquad \forall \xi=(\xi_l)\in\R^d,
$$
and there are $K_R',K_R''<\infty$ such that for all $x\in\varOmega$
$$
|f_{k,l}(\nabla u^*(x))| \le K_R', \quad |f_{k,lm}(\nabla u^*(x))| \le K_R''.
$$
Hence we have
$$
I_1
\le -\kappa_R \| e \|^2 + d K_R' \| e \|\, \|d \| \le -\tfrac12 \kappa_R \| e \|^2 + C_R \| d \|^2
$$
with a suitable $C_R$ (which depends on $\kappa_R$ and $K_R'$).
By the regularity condition \eqref{Regularity}, also the $W^{2,q}$ norm of $u^*$ is bounded. We obtain with H\"older's inequality, for $r$ such that $\frac12 +\frac1q +\frac1r=1$,
$$
I_2 \le  K_R'' \, \| \nabla e \|_{L^2} \, \| u^* \|_{W^{2,q}} \, \| e+d \|_{L^r}.
$$
For $q>d$, we have $2d/(d-2)>2q/(q-2)=r>2$, and so 
$H^1_0(\varOmega)$ is compactly embedded into 
$L^r(\varOmega)$.  
In this case, for every $\eps>0$ there exists $C_\eps<\infty$ such that for all $v\in H^1_0(\varOmega)$,
$$
\| v \|_{L^r} \le \eps \| \nabla v \|_{L^2} + C_\eps \| v \|_{L^2}.
$$
We conclude that there exists a constant $C$ such that
$$
I_1+I_2 \le - \tfrac13  \kappa_R \| e \|^2 + C \| d \|^2.
$$
To bound the second term on the right-hand side of \eqref{eed} we use Theorem~\ref{thm:main}.
This ensures us that the ${W^{1,\infty}(\varOmega)}$ and the $W^{2,q}(\varOmega)$ norm of the numerical approximation are bounded (independently of $\tau$), and so we can use the local Lipschitz continuity of $f_{k,l}$ and an $W^{2,q}(\varOmega)$ bound of the numerical solution. We thus obtain, for arbitrary $\eps>0$ and for $r$ such that $\frac12 +\frac1q +\frac1r=1$,
\begin{align*}
\Bigl\langle \sum_{k,l=1}^d \bigl( f_{k,l}(\nabla u) &- f_{k,l}(\nabla u^*) \bigr)  \partial_k\partial_l u, e+d \Bigr\rangle 
\\[-2mm]
&\le \widetilde C_R \,\| \nabla e \|_{L^2(\varOmega)} \, \| u \|_{W^{2,q}(\varOmega)} \, \| e+d \|_{L^r(\varOmega)} \\
&\le \eps \|  e \|^2+ C_\eps \| d \|^2.
\end{align*}
Combining the above bounds thus yields (taking up again the dropped subscripts $n,i$)
$$
\langle \dot e_{n,i}, e_{n,i} + d_{n,i} \rangle \le -\tfrac14\kappa_R  \| e_{n,i} \|^2 + C \| d_{n,i} \|^2.
$$
With the same arguments, again invoking Theorem~\ref{thm:main}, we also obtain from \eqref{err-a} (with a different constant $C$)
$$
\| \dot e_{n,i} \|_* \le C \, \| e_{n,i} \|.
$$
The second term in \eqref{err-square} is estimated as (note that all $b_i>0$)
\begin{align*}
\bigl\langle d_{n+1}, e_{n} + \tau \sum_{i=1}^s b_{i} \dot e_{n,i} \bigr\rangle \le
 \sqrt\tau\| d_{n+1}/\tau \|_*\, \sqrt\tau\| e_n \| + \sqrt\tau\| d_{n+1} \|\, \sqrt\tau \sum_{i=1}^s b_i \| \dot e_{n,i} \|_*,
\end{align*}
and 
$$
|d_{n+1}|^2 \le \sqrt\tau\|d_{n+1}/\tau\|_*\, \sqrt\tau\|d_{n+1}\| \le \tfrac12 \tau\, \|d_{n+1}\|^2 +\tfrac12 \tau\, \|d_{n+1}/\tau\|_*^2\,.
$$
Combining the above estimates (and noting that $e_n=e_{n-1,s}$) we obtain
\begin{align*}
 |e_{n+1}|^2 - |e_n|^2 &+ \tfrac18\kappa_R \tau \sum_{i=1}^s b_i \|e_{n,i}\|^2 \\
 &\le
 C\tau\sum_{i=1}^s \| d_{n,i}\|^2 + C\tau (\|d_{n+1}\|^2 + \|d_{n+1}/\tau\|_*^2).
\end{align*}
Summing up these inequalities and recalling \eqref{delta} yields
$$
|e_{n+1}|^2 + \tfrac18\kappa_R \tau \sum_{m=0}^n \sum_{i=1}^s b_i \|e_{m,i}\|^2 \le C\delta,
$$
which completes the proof.\QED

%
%
%
%
%

%

\bigskip

{\bf Acknowledgements.} 
The research stay of Buyang Li at the University of T\" ubingen
is funded by the Alexander von Humboldt Foundation. Peer Kunstmann and Christian Lubich are supported by Deutsche Forschungsgemeinschaft, SFB 1173. We thank Christian Power Guerra for the interpretation of \eqref{triangle-inequality} as a triangle inequality.

\bibliographystyle{amsplain}

\end{document}